\documentclass[12pt]{amsart}
\usepackage{amssymb}
\usepackage{graphicx}
\usepackage{bm}

\newcommand{\eps}{\varepsilon}
\renewcommand{\phi}{\varphi}
\renewcommand{\kappa}{\varkappa}
\renewcommand{\setminus}{\smallsetminus}
\newcommand{\xhat}{\bm\hat{x}}
\newcommand{\yhat}{\bm\hat{y}}
\newcommand\blfootnote[1]{%
  \begingroup
  \renewcommand\thefootnote{}\footnote{#1}%
  \addtocounter{footnote}{-1}%
  \endgroup
}

\theoremstyle{plain}
\newtheorem{theorem}{Theorem}[section]
\newtheorem{corollary}[theorem]{Corollary}
\newtheorem{lemma}[theorem]{Lemma}
\newtheorem{proposition}[theorem]{Proposition}
\newtheorem{conjecture}[theorem]{Conjecture}
\theoremstyle{definition}
\newtheorem{definition}[theorem]{Definition}

\newtheorem{problem}[theorem]{Problem}

\usepackage[colorlinks=true,linkcolor=blue]{hyperref}

\begin{document}

\title{Family of chaotic maps from game theory}

\author[T. Chotibut]{Thiparat Chotibut}
\address[T. Chotibut]{Engineering Systems and Design, Singapore
  University of Technology and Design, 8 Somapah Road, Singapore 487372}
\email{thiparatc@gmail.com, thiparat\_chotibut@sutd.edu.sg}

\author[F. Falniowski]{Fryderyk Falniowski}
\address[F. Falniowski]{Department of Mathematics, Cracow University
  of Economics, Ra\-ko\-wicka~27, 31-510 Krak\'ow, Poland}
\email{fryderyk.falniowski@uek.krakow.pl}

\author[M. Misiurewicz]{Micha{\l} Misiurewicz}
\address[M. Misiurewicz]{Department of Mathematical Sciences, Indiana
  University-Purdue University Indianapolis, 402 N. Blackford
  Street, Indianapolis, IN 46202, USA}
\email{mmisiure@math.iupui.edu}

\author[G. Piliouras]{Georgios Piliouras}
\address[G. Piliouras]{Engineering Systems and Design, Singapore
  University of Technology and Design, 8 Somapah Road, Singapore 487372}
\email{georgios@sutd.edu.sg}

\maketitle


\begin{abstract}
From a two-agent, two-strategy congestion game where both agents apply
the multiplicative weights update
algorithm, we obtain a two-parameter family of maps of the unit square
to itself. Interesting dynamics arise on the invariant diagonal, on
which a two-parameter family of bimodal interval maps exhibits
periodic orbits and chaos. While the fixed point $b$ corresponding to
a Nash equilibrium of such map $f$ is usually repelling, it is
globally \emph{Ces\`{a}ro attracting} on the diagonal, that is,
\[
\lim_{n\to\infty}\frac1n\sum_{k=0}^{n-1}f^k(x)=b
\]
for every $x$ in the minimal invariant interval. This solves a known open
question whether there exists a nontrivial smooth map other than
$x\mapsto axe^{-x}$ with centers of mass of all periodic orbits
coinciding. We also study the dependence of the dynamics on the two
parameters.
\end{abstract}


\section{Introduction}\label{s:intro}

We will consider a $2\times2$ (two agents/players, two pure strategies)
\emph{congestion game} (see~\cite{ros,AGT}) and examine its dynamics under
a specific model of learning behavior for both agents known as
\emph{multiplicative weights update}~\cite{Arora05themultiplicative}.

We will start off with the description of the two agent game. Let $x$
be the probability that the first player (agent) chooses the first
strategy (he chooses the second strategy with probability $1-x$).
Similarly, let $y$ be the probability that the second player (agent)
chooses the first strategy (he chooses the second strategy with
probability $1-y$). We denote a tuple of such randomized strategies as
$(x,y)$. If the agents choose the same strategy, this leads to
congestion, and the cost increases. We will assume that the cost of a
strategy is proportional to its load, i.e., to the number of agents
choosing this strategy. If we denote by $c(i,j)$ the expected cost of
the player number $i$ playing the strategy number $j$, and the
coefficients of proportionality are $\alpha,\beta$, then we get
\begin{equation}\label{cost}
\begin{aligned}
c(1,1)&=\alpha(1+y), & c(2,1)&=\alpha(1+x),\\
c(1,2)&=\beta(1+(1-y))=\beta(2-y),\ \ & c(2,2)&=
\beta(1+(1-x))=\beta(2-x).
\end{aligned}
\end{equation}

A strategy profile/tuple $(x,y)$ is a \emph{Nash equilibrium} if and
only if no agent can strictly decrease their expected cost by
unilaterally deviating to another strategy.

\subsection{Multiplicative Weights Update}

We will assume that both agents participate in the game by applying
the \emph{multiplicative weights update} (MWU) algorithm
\cite{Arora05themultiplicative}. At time $n$, the first (second)
player/agent chooses the first strategy with probability $x_n$
($y_n$). We will study MWU in the full informational setting, where
each agent gets to experience the cost of all strategies available to
them. Moreover, the (expected) cost of strategy j for agent $1$
(respectively, $2$) at time $n$ is computed given the randomized
choice of agent $2$ (respectively, $1$). In this case, the update of
the mixed/randomized strategies of each agent when applying MWU with
learning rate $\eps\in(0,1)$ is as follows:
\begin{equation}\label{mwu}
\begin{aligned}
x_{n+1}&=\frac{x_n(1-\eps)^{c(1,1)}}{x_n(1-\eps)^{c(1,1)}+
(1-x_n)(1-\eps)^{c(1,2)}},\\
y_{n+1}&=\frac{y_n(1-\eps)^{c(2,1)}}{y_n(1-\eps)^{c(2,1)}+
(1-y_n)(1-\eps)^{c(2,2)}}.
\end{aligned}
\end{equation}

Observe that in this way large cost at time $n$ leads to the decrease
of the probability of doing the same at time $n+1$. The learning rate
$\eps$ can be thought of as capturing the patience of the agents. For
small $\eps$ the agents adapt slowly to the costs whereas for large
$\eps$ they respond more aggressively. We will study the effects of
this learning rate $\eps$ to the stability of the MWU dynamics.
In~\cite{PPP} it was shown that there exist $2 \times 2$ congestion
games where MWU with large enough $\eps$ can lead to limit cycles or
chaos. We will establish the emergence of chaos for all $2 \times 2$
congestion games for large enough learning rate $\eps$.

\section{Dynamical model}\label{s:model}

Let us plug into~\eqref{mwu} the values of the cost functions
from~\eqref{cost}:
\begin{equation}\label{mwu1}
\begin{aligned}
x_{n+1}&=\frac{x_n(1-\eps)^{\alpha(1+y_n)}}{x_n(1-\eps)^{\alpha(1+y_n)}+
(1-x_n)(1-\eps)^{\beta(2-y_n)}}\\
&=\frac{x_n}{x_n+(1-x_n)(1-\eps)^{\beta(2-y_n)-\alpha(1+y_n)}},\\
y_{n+1}&=\frac{y_n(1-\eps)^{\alpha(1+x_n)}}{y_n(1-\eps)^{\alpha(1+x_n)}+
(1-y_n)(1-\eps)^{\beta(2-x_n)}}\\
&=\frac{y_n}{y_n+(1-y_n)(1-\eps)^{\beta(2-x_n)-\alpha(1+x_n)}}.
\end{aligned}
\end{equation}

We introduce new variables
\begin{equation}\label{var}
a=(\alpha+\beta)\ln\frac1{1-\eps},\ \ \ b=\frac{2\beta-\alpha}{\alpha+\beta}.
\end{equation}
Then formulas~\eqref{mwu1} become
\begin{equation}\label{map}
\begin{aligned}
x_{n+1}&=\frac{x_n}{x_n+(1-x_n)\exp(a(y_n-b))},\\
y_{n+1}&=\frac{y_n}{y_n+(1-y_n)\exp(a(x_n-b))}.
\end{aligned}
\end{equation}

Clearly, if $(x_n,y_n)\in[0,1]^2$ then also
$(x_{n+1},y_{n+1})\in[0,1]^2$. Therefore we will be studying the
family of maps $F_{a,b}:[0,1]^2\to[0,1]^2$, given by
\begin{equation}\label{map2d}
F_{a,b}(x,y)=\left(\frac{x}{x+(1-x)\exp(a(y-b))},
\frac{y}{y+(1-y)\exp(a(x-b))}\right).
\end{equation}

The diagonal $x=y$ is invariant for $F_{a,b}$, so we can restrict this
map to the diagonal and we get the family of maps
$f_{a,b}:[0,1]\to[0,1]$, given by
\begin{equation}\label{map1d}
f_{a,b}(x)=\frac{x}{x+(1-x)\exp(a(x-b))}.
\end{equation}
In the context of game theory, this restriction means that both players start
with the same mixed strategy (the same probability distributions).

Our aim is to investigate the long-term behavior of the orbits of
$F_{a,b}$ and $f_{a,b}$. In particular, we can consider periodic
orbits, their stability, or chaos. When speaking of chaos, we will use
its most popular kind, \emph{Li-Yorke chaos}. Namely, if $X$ is a compact
space with metric $\rho$ and $f:X\to X$ is a continuous map, we say
that the pair of points $x,y\in X$ is a \emph{Li-Yorke pair} if
\begin{align*}
\liminf_{n\to\infty}\rho(f^n(x),f^n(y))&=0,\\
\limsup_{n\to\infty}\rho(f^n(x),f^n(y))&>0.
\end{align*}
The map $f$ is \emph{Li-Yorke chaotic} if there is an uncountable set
$S\subset X$ (called \emph{scrambled set}) such that every pair
$(x,y)$ with $x,y\in S$ and $x\neq y$ is a Li-Yorke pair.

Before starting detailed investigation of the families of maps
$F_{a,b}$ and $f_{a,b}$, let us determine what we should assume about
$a$ and $b$. By~\eqref{var}, $a>0$ and as we mentioned, we will be
interested most at large values of $a$. While~\eqref{var} does not
provide any restrictions for $b$, let us think what happens if $b<0$.

Then $F_{a,b}$ has four fixed points: $(0,0), (0,1), (1,0), (1,1)$.
Moreover, $x_{n+1}<x_n$ unless $x_n=0$ or $x_n=1$. Similarly,
$y_{n+1}<y_n$ unless $y_n=0$ or $y_n=1$. Therefore, trajectories of
all points of $(0,1)^2$ converge to $(0,0)$. Similarly, if $b>1$ then
trajectories of all points of $(0,1)^2$ converge to $(1,1)$. This
dynamics is not interesting. If $b=0$ or $b=1$, there are additionally
segments of fixed points on the boundary of the square, but still the
dynamics is not worth studying. Therefore, we will assume that
$b\in(0,1)$.

\section{On the diagonal}\label{s:1d}

Let us start investigating the dynamics of the maps $f_{a,b}$, given
by~\eqref{map1d}, with $a>0$, $b\in (0,1)$ (see Figure~\ref{graph}).
This map has three fixed points: $0$, $b$ and $1$.

\begin{figure}
\begin{center}
\includegraphics[height=60truemm]{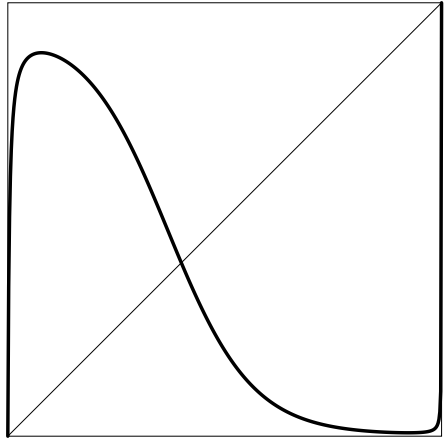}
\caption{The map $f_{a,b}$ with $a=14$, $b=0.4$.}\label{graph}
\end{center}
\end{figure}

The derivative of $f_{a,b}$ is given by
\begin{equation}\label{der}
f'_{a,b}(x)=\frac{(ax^2-ax+1)\exp(a(x-b))}
{\big(x+(1-x)\exp(a(x-b))\big)^2}.
\end{equation}
Thus,
\[
f'_{a,b}(0)=\exp(ab),\ \ f'_{a,b}(1)=\exp(a(1-b)),\ \ f'_{a,b}(b)=ab^2-ab+1.
\]
We see that the fixed points 0 and 1 are always repelling, while $b$
is repelling if $a>\frac2{b(1-b)}$.

The critical points of $f_{a,b}$ are solutions to $ax^2-ax+1=0$. Thus,
if $0<a\le 4$, then $f_{a,b}$ is strictly increasing. If $a>4$, it has
two critical points
\begin{equation}\label{crit}
c_l=\frac12-\sqrt{\frac14-\frac1a},\ \ \
c_r=\frac12+\sqrt{\frac14-\frac1a},
\end{equation}
so the map $f_{a,b}$ is bimodal.

Our family has some symmetry. Consider the flip $\phi:[0,1]\to[0,1]$,
$\phi(x)=1-x$. Then
\begin{align*}
\phi(f_{a,b}(x))&=1-\frac{x}{x+(1-x)\exp(a(x-b))}
=\frac{(1-x)\exp(a(x-b))}{x+(1-x)\exp(a(x-b))}\\
&=\frac{(1-x)}{(1-x)+x\exp(a((1-x)-(1-b)))}=f_{a,1-b}(\phi(x)).
\end{align*}
This can be written as
\begin{equation}\label{sym}
\phi\circ f_{a,b}=f_{a,1-b}\circ\phi,\ \ \textrm{where}\ \ \phi(x)=1-x.
\end{equation}

While $[0,1]$ is the natural space on which $f_{a,b}$ acts, it will be
sometimes easier to think of a smaller interval.

\begin{lemma}\label{l2}
For every $a>0$, $b\in (0,1)$ there exists a closed interval
$I_{a,b}\subset(0,1)$ such that $f_{a,b}(I_{a,b})\subset I_{a,b}$ and
$f_{a,b}$-trajectories of all points of $(0,1)$ enter $I_{a,b}$.
\end{lemma}

\begin{proof}
Since $0$ and $1$ are repelling fixed points of $f_{a,b}$, there exist
$\delta_1>0$ and $\lambda>1$ such that if $x<\delta_1$ or $1-x<\delta_1$
then $\lambda x<f_{a,b}(x)<1-\lambda x$. We have
$f_a([\delta_1,1-\delta_1])\subset(0,1)$, so there exists $\delta_2>0$ such
that $f_a([\delta_1,1-\delta_1])\subset(\delta_2,1-\delta_2)$. Set
$\delta=\min(\delta_1,\delta_2)$ and $I_{a,b}=[\delta,1-\delta]$. Then $I_{a,b}$
is mapped by $f_{a,b}$ to itself, and all $f_{a,b}$-trajectories of
points from $(0,1)$ sooner or later enter $I_{a,b}$.
\end{proof}

Let us investigate regularity of $f_{a,b}$. It is clear that it is
analytic. However, nice properties of interval maps are guaranteed not
by analyticity, but by the negative Schwarzian derivative. Let us recall
that the Schwarzian derivative of $f$ is given by the formula
\[
Sf=\frac{f'''}{f'}-\frac32\left(\frac{f''}{f'}\right)^2.
\]
A ``metatheorem'' states that almost all natural noninvertible
interval map have negative Schwarzian derivative. Note that if $a\le
4$ then $f_{a,b}$ is a homeomorphism, so we should not expect negative
Schwarzian derivative for that case (and it would be useless, anyway).

\begin{proposition}\label{nS}
If $a>4$ then the map $f_{a,b}$ has negative Schwarzian derivative.
\end{proposition}

\begin{proof}
Instead of making very complicated computations, we will use the
formula $S(h\circ f)=(f')^2((Sh)\circ f)+Sf$ and the fact that
M\"obius transformations have zero Schwarzian derivative. Set
$h(x)=\exp(ab)\frac{1-x}{x}$ and $g(x)=h\circ
f_{a,b}(x)=\frac{1-x}{x}\exp(ax)$; then $Sg=Sf$.

Now, simple computations yield
\[
Sg(x)=-\frac{a\big(-12+6a+4a^2(-1+x)x+a^3(-1+x)^2x^2\big)}
{2\big(1+a(-1+x)x\big)^2}.
\]
Thus, the Schwarzian derivative of $f_{a,b}$ is negative if and only
if
\[
-12+6a+4a^2x(x-1)+a^3x^2(1-x)^2>0.
\]
Set $t=x(1-x)$. Then the above inequality becomes
\begin{equation}\label{ineq}
P_a(t)>0,\ \ \ \textrm{where}\ \ \ P_a(t)=a^3t^2-4a^2t+6a-12.
\end{equation}
Clearly, $t\in [0,1/4]$.

The quadratic polynomial $P_a$ attains its minimum at $t=2/a$. If
$2/a\leq 1/4$, that is, if $a\geq 8$, we have $P_a(2/a)=2a-12>0$,
so~\eqref{ineq} holds. If $a\in (0,8)$, then~\eqref{ineq} holds
whenever $P_a(1/4)>0$. We have $P_a(1/4)=(a-4)(\frac{1}{16}a^2-\frac
34a+3)$. The second factor is always positive, and thus $P_a(1/4)>0$
for $a>4$.
\end{proof}

For maps with negative Schwarzian derivative each attracting or
neutral periodic orbit has a critical point in its immediate basin of
attraction. Thus, we know that if $a>4$ then $f_{a,b}$ can have at
most two attracting or neutral periodic orbits.

\subsection{Average behavior}

While we know that the fixed point $b$ is often repelling, especially
for large values of $a$, we can show that it is attracting in a
certain sense.

\begin{definition}\label{d:Cesaro}
For an interval map $f$ a point $p$ is \emph{Ces\`{a}ro attracting} if it
has a neighborhood $U$ such that for every $x\in U$ the averages
\[
\frac1n\sum_{k=0}^{n-1}f^k(x)
\]
converge to $p$.
\end{definition}

We will show that $b$ is globally Ces\`{a}ro attracting. Here by
``globally'' we mean that the set $U$ from the definition is the
interval $(0,1)$.

\begin{theorem}\label{t:Cesaro}
For every $a>0$, $b\in(0,1)$ and $x\in(0,1)$ we have
\begin{equation}\label{e:Cesaro}
\lim_{n\to\infty}\frac1n\sum_{k=0}^{n-1}f_{a,b}^k(x)=b.
\end{equation}
\end{theorem}

\begin{proof}
By Lemma~\ref{l2} there is a closed interval $I_{a,b}\subset (0,1)$
which is invariant for $f_{a,b}$. Thus, there is $\delta\in(0,1)$ such
that $I_{a,b}\subset (\delta,1-\delta)$.

Fix $x=x_0\in[0,1]$ and use our notation $x_n=f_{a,b}^n(x_0)$. By induction,
we get
\begin{equation}\label{iter}
x_n=\frac{x}{x+(1-x)\exp\left(a\sum_{k=0}^{n-1}(x_k-b)\right)}.
\end{equation}

Assume that $x=x_0\in I_{a,b}$. Since $\delta<x_n<1-\delta$, we have
\[
\frac{x}{1-\delta}<x+(1-x)\exp\left(a\sum_{k=0}^{n-1}(x_k-b)\right)
<\frac{x}{\delta},
\]
so
\[
\delta^2<x\frac{\delta}{1-\delta}<(1-x)\exp\left(a\sum_{k=0}^{n-1}(x_k-b)
\right)<x\frac{1-\delta}{\delta}<\frac1\delta.
\]
Therefore
\begin{equation}\label{est1}
\delta^2<\exp\left(a\sum_{k=0}^{n-1}(x_k-b) \right)<\frac1{\delta^2},
\end{equation}
so
\[
\left|a\sum_{k=0}^{n-1}(x_k-b)\right|<2\log(1/\delta).
\]
This inequality can be rewritten as
\[
\left|\frac1n\sum_{k=0}^{n-1}x_k-b\right|<\frac{2\log(1/\delta)}{an},
\]
and~\eqref{e:Cesaro} follows.

If $x\in(0,1)\setminus I_{a,b}$, then by the definition of $I_{a,b}$
(Lemma~\ref{l2}) there is $n_0$ such that $f_{a,b}^{n_0}(x)\in
I_{a,b}$, so~\eqref{e:Cesaro} also holds.
\end{proof}

\begin{corollary}\label{cmper}
For every periodic orbit $\{x_0,x_1,\dots,x_{n-1}\}$
of $f_{a,b}$ in $(0,1)$ its center of mass
\[
\frac{x_0+x_1+\dots+x_{n-1}}n
\]
is equal to $b$.
\end{corollary}

Applying the Birkhoff Ergodic Theorem, we get a stronger corollary.

\begin{corollary}\label{cmmeas}
For every probability measure $\mu$, invariant for $f_{a,b}$ and such
that $\mu(\{0,1\})=0$, we have
\[
\int_{[0,1]} x\; d\mu=b.
\]
\end{corollary}

Corollary~\ref{cmper} solves a known question whether there are
nontrivial smooth maps other than $x\mapsto axe^{-x}$ with centers of
mass of all periodic orbits coinciding~\cite{rt}. Here ``nontrivial'' means that
there are periodic orbits of infinitely many periods; we will show
later that for many parameters $f_{a,b}$ has this property.

\begin{problem}\label{findall}
Find all nontrivial smooth (analytic) maps for which centers of mass
of all periodic orbits coincide.
\end{problem}

\subsection{Symmetric case}
Now we explore what happens as we fix $b$ and let $a$ go to
infinity. First we study the symmetric case, $b=1/2$, which
corresponds to equal coefficients of the cost functions,
$\alpha=\beta$. To simplify notation we denote $f_a=f_{a,1/2}$, so
\[
f_a(x)=\frac{x}{x+(1-x)\exp(a(x-1/2))}.
\]

The interpretation of the formula~\eqref{sym} is very simple in this
case: the maps $f_a$ and $\phi$ commute. Set $g_a=\phi\circ
f_a=f_a\circ\phi$. Since $\phi$ is an involution, we have $g_a^2=f_a^2$.

We show that the dynamics of $f_a$ is simple, no matter how large $a$
is.

\begin{theorem}\label{trajgf}
The $g_a$-trajectory of every point of $(0,1)$ converges to a fixed
point of $g_a$. The $f_a$-trajectory of every point of $(0,1)$
converges to a fixed point or a periodic orbit of period $2$ of $f_a$,
other than $0$ and $1$.
\end{theorem}

\begin{proof}
We want to find fixed points and points of period 2 of $f_a$ and
$g_a$. Clearly,
\[
f_a(0)=0,\ \ \ f_a(1)=1,\ \ \ g_a(0)=1,\ \ \ g_a(1)=0.
\]

By~\eqref{iter} we have
\[
f_a^2(x)=\frac{x}{x+(1-x)\exp(a(x+f_a(x)-1))},
\]
so the fixed points of $f_a^2$ are $0$, $1$ and the solutions to
$x+f_a(x)-1=0$, that is, to $g_a(x)=x$. Thus, the fixed points of
$g_a^2$ (which, as we noticed, is equal to $f_a^2$) are the fixed
points of $g_a$ and $0$ and $1$.

We can choose the invariant interval $I_a=I_{a,1/2}$ symmetric, so
that $\phi(I_a)=I_a$. Let us look at $G_a=g_a|_{I_a}:I_a\to I_a$. All
fixed points of $G_a^2$ are also fixed points of $G_a$, so $G_a$ has
no periodic points of period 2. By the Sharkovsky Theorem, $G_a$ has
no periodic points other than fixed points. For such maps it is known
(see, e.g.,~\cite{SKSF}) that the $\omega$-limit set of every
trajectory is a singleton of a fixed point, that is, every trajectory
converges to a fixed point. If $x\in(0,1)\setminus I_a$, then the
$g_a$-trajectory of $x$ after a finite time enters $I_a$, so
$g_a$-trajectories of all points of $(0,1)$ converge to a fixed point
of $g_a$ in $I_a$. Observe that a fixed point of $g_a$ can be a fixed
point of $f_a$ (other that 0, 1) or a periodic point of $f_a$ of
period 2. Thus the $f_a$-trajectory of every point of $(0,1)$
converges to a fixed point or a periodic orbit of period $2$ of $f_a$,
other than $0$ and $1$.
\end{proof}

Let us identify the possible limits from the above theorem.

\begin{theorem}\label{trajf}
If $0<a\le 8$ then $f_a$-trajectories of all points of $(0,1)$
converge to the fixed point $1/2$. If $a>8$ then $f_a$ has a periodic
attracting orbit $\{\sigma_a,1-\sigma_a\}$, where $0<\sigma_a<1/2$. This
orbit attracts trajectories of all points of $(0,1)$, except countably
many points, whose trajectories eventually fall into the repelling
fixed point $1/2$.
\end{theorem}

\begin{proof}
Observe first that $1/2$ is a fixed point of both $f_a$ and $g_a$.
Then, let us look for the fixed points of $g_a$ in $[0,1/2]$. They are
the solutions of the equation $g_a(x)=x$, that is, $f_a(x)=1-x$, which
is equivalent to
\[
x^2=(1-x)^2\exp(a(x-1/2)).
\]
Since $x$, $1-x$ and $\exp(a(x-1/2))$ are non-negative, this equation
is equivalent to $\gamma_a(x)=0$, where
\[
\gamma_a(x)=x-(1-x)\exp((a/2)(x-1/2)).
\]
We have
\[
\begin{aligned}
\gamma_a'(x)&=1+\big(1-(a/2)+(a/2)x\big)\exp\big((a/2)(x-1/2)\big),\\
\gamma_a''(x)&=(a/2)\big(2-(a/2)+(a/2)x\big)\exp\big((a/2)(x-1/2)\big).
\end{aligned}
\]
For all $a>0$ we have $\gamma_a(0)<0$ and $\gamma_a(1/2)=0$. If $a\ge
8$ then $\gamma_a''\le 0$ on $[0,1/2]$, so $\gamma_a$ is concave
there. We have $\gamma_8'(1/2)=0$, so since $\gamma_a$ is real
analytic, we have $\gamma_8(x)<0$ for all $x\in[0,1/2)$. For all
$a\in(0,8]$ and $x\in[0,1/2)$ we have $\gamma_a(x)\le\gamma_8(x)<0$,
so in this case $g_a$ has no fixed points in $[0,1/2)$. However, if
$a>8$ then $\gamma_a'(1/2)=2-a/4<0$, so $\gamma_a$ has exactly one
zero in $[0,1/2)$. Thus, in this case $g_a$ has exactly one fixed
point in $[0,1/2)$. We denote it by $\sigma_a$.

The situation for $x\in(1/2,1]$ is the same, because $\phi$ conjugates
$g_a$ with itself and maps $[0,1/2)$ to $(1/2,1]$. Observe that
$f_a(\sigma_a)=1-g_a(\sigma_a)=1-\sigma_a$ and similarly,
$f_a(1-\sigma_a)=\sigma_a$.

Now, if $0<a\le 8$, then $1/2$ is the only fixed point of $g_a$ in
$(0,1)$, so by Theorem~\ref{trajgf}, the $g_a$-trajectory of every
point of $(0,1)$ converges to $1/2$. Thus, the $f_a$-trajectory of
every point of $(0,1)$ also converges to $1/2$.

If $a>8$, then $|g_a'(1/2)|=|f_a'(1/2)|>1$ by~\eqref{der}, so the only
way a $g_a$-trajectory of $x$ can converge to $1/2$ is that
$g_a^n(x)=1/2$ for some $n$. Since the function $g_a$ is real
analytic, there are only countably many such points $x$. According to
Theorem~\ref{trajgf}, the $g_a$-trajectories of all other points of
$(0,1)$ converge to $\sigma_a$ or $1-\sigma_a$. Thus, the period 2 orbit
$\{\sigma_a,1-\sigma_a\}$ of $f_a$ attracts $f_a$-trajectories of all
points of $(0,1)$, except countably many points, whose trajectories
eventually fall to the repelling fixed point $1/2$.
\end{proof}

\subsection{Asymmetric case}

Now we proceed with the case when $b\neq 1/2$, that is, the
cost functions differ. As we noticed in Section~\ref{s:1d}, if $a\le
4$ then $f_{a,b}$ is strictly increasing and has three fixed points: 0
and 1 repelling and $b$ attracting. Therefore in this case
trajectories of all points of $(0,1)$ converge in a monotone way to
$b$.

We know that the fixed point $b$ is repelling if and only if
$a>\frac2{b(1-b)}$.

\begin{conjecture}\label{battr}
If $a\le\frac2{b(1-b)}$ then trajectories of all points of $(0,1)$
converge to $b$.
\end{conjecture}

This is a simple situation, so we turn to the case of large $a$. We
fix $b\in(0,1)\setminus\{1/2\}$ and let $a$ go to infinity. We will
show that if $a$ becomes sufficiently large (but how large, depends on
$b$), then $f_{a,b}$ is Li-Yorke chaotic and has periodic orbits of
all periods.

\begin{theorem}\label{per3}
If $b\in (0,1)\setminus\{1/2\}$, then there exists $a_b$ such that if
$a>a_b$ then $f_{a,b}$ has periodic orbit of period 3.
\end{theorem}

\begin{proof}
Fix $a>0$ and $b,x\in(0,1)$ (we will vary $a$ later). As in the proof
of Theorem~\ref{t:Cesaro}, we set $x_n=f_{a,b}^n(x)$, and then the
formula~\eqref{iter} holds. Hence, $f_{a,b}(x)>x$ is equivalent to
$x<b$ and $f_{a,b}^3(x)<x$ is equivalent to $x+f_{a,b}(x)+f_{a,b}^2(x)
>3b$.

Assume that $0<b<1/2$. Then $3b-1<b$, so we can take $x>0$ such that
$3b-1<x<b$. Then $f_{a,b}(x)>x$. Moreover, $\exp(a(x-b))$ goes to 0 as
$a$ goes to infinity, so
\[
\lim_{a\to\infty}f_{a,b}(x)=\lim_{a\to\infty}
\frac{x}{x+(1-x)\exp(a(x-b))}=1.
\]
Thus, since $3b-x<1$, there exists $a_b$ such that if $a>a_b$ then
$f_{a,b}(x)>3b-x$, so $x+f_{a,b}(x)+f_{a,b}^2(x)>3b$ (note that $x$
depends only on $b$, not on $a$). Hence, if $a>a_b$ then
$f_{a,b}^3(x)<x$.

Now, because for $a>a_b$ there exist $x$ such that
$f_{a,b}^3(x)<x<f_{a,b}(x)$, from the theorem from~\cite{LMPY} it
follows that if $a>a_b$ then $f_{a,b}$ has a periodic point of period
3.

If $1/2<b<1$ then by~\eqref{sym} we can reduce it to the case
$0<b<1/2$.
\end{proof}

By the Sharkovsky Theorem (\cite{sha}, see also~\cite{LY}), existence
of a periodic orbit of period 3 implies existence of periodic orbits
of all periods, and by the result of~\cite{LY}, it implies that the
map is Li-Yorke chaotic.

Thus, we get the following corollary.

\begin{corollary}\label{chaos}
If $b\in (0,1)\setminus\{1/2\}$, then there exists $a_b$ such that if
$a>a_b$ then $f_{a,b}$ has periodic orbits of all periods and is
Li-Yorke chaotic.
\end{corollary}

This result has an interesting interpretation in the context of game
theory and learning in games. Parameter $a$ can be treated as measuring the
aggressiveness of a player. Corollary~\ref{chaos} implies that if
players have different cost functions (but they do not differ too much,
that is, $\alpha<2\beta$ and $\beta<2\alpha$ in \eqref{cost}),
their behavior will be chaotic if only they are aggressive enough.
Moreover, having in mind the connection between this parameter and
learning rate $\eps$, see~\eqref{var}, we can interpret
Corollary~\ref{chaos} as stating that if players learn fast enough ($\eps$ is
close to 1) then the system may become chaotic.

\section{Off the diagonal}\label{s:2d}

The natural question which arises is whether the chaotic behavior emerges
only if both players share the same initial strategy, or chaos can
happen also when their initial strategy profiles are different. To
answer this question, we study the
family of maps $F_{a,b}\colon [0,1]^2\to [0,1]^2$ defined
by~\eqref{map2d}, with $a>0$ and $b\in (0,1)$.

The map $F_{a,b}$ has five fixed points: $(0,0)$, $(1,1)$, $(0,1)$, $(1,0)$ and
$(b,b)$. Three of them, namely $(0,1)$,
$(1,0)$ and $(b,b)$, are Nash equilibria of the congestion game. That
is, no agent can strictly decrease their cost by unilateral deviations
to another strategy.

Indeed, in the case of the $(b,b)$ randomized strategy profile, the
expected cost of both agents (in the game defined by (\ref{cost})) is
equal to $3\alpha\beta/(\alpha+\beta)$. Moreover, if the first
(second) agent were to deviate and choose a different strategy, his
expected cost would still remain
unchanged. This means that $(b,b)$ is a Nash equilibrium of the game. By our
assumption that $0<b<1$ we derive that $\alpha < 2 \beta$ and $\beta <
2 \alpha$. This implies that $(0,0)$ and $(1,1)$ are not Nash
equilibria, as each agent would strictly prefer to choose distinct
strategies than use the same strategy as the other agent. Applying the
same reasoning, states $(0,1)$ and $(1,0)$ are Nash equilibria, since
in these states deviating to another strategy leads to sharing the
same strategy as your opponent.

The derivative of $F_{a,b}$ at $(0,1)$ and $(1,0)$ is
\[
\begin{aligned}
DF_{a,b}(0,1)&=\begin{pmatrix} \exp(-a(1-b)) & 0\\ 0 & \exp(-ab)
\end{pmatrix},\\
DF_{a,b}(1,0)&=\begin{pmatrix} \exp(-ab) & 0\\ 0 & \exp(-a(1-b))
\end{pmatrix},
\end{aligned}
\]
so those points are attracting.

To study the behavior of $F_{a,b}$ close to the other three fixed
points, and in general, close to the diagonal, we compute the
derivative of $F_{a,b}$ on the diagonal:
\begin{equation}\label{derdiag}
DF_{a,b}(x,x)=\frac{\exp(a(x-b))}{\big(x+(1-x)\exp(a(x-b))\big)^2}
\begin{pmatrix} 1 & -ax(1-x)\\ -ax(1-x) & 1 \end{pmatrix}.
\end{equation}
Thus, the eigenvectors of $DF_{a,b}$ at points $(x,x)$ on the diagonal
are $\begin{pmatrix} 1 \\ 1 \end{pmatrix}$ and $\begin{pmatrix} 1
  \\ -1 \end{pmatrix}$. The eigenvalue corresponding to the first one
is of course the derivative of $f_{a,b}$ at $x$, given by the
formula~\eqref{der}. The eigenvalue corresponding to the second
vector, perpendicular to the diagonal, is
\begin{equation}\label{derper}
\lambda_{a,b}(x)=\frac{(1+ax(1-x))\exp(a(x-b))}
{\big(x+(1-x)\exp(a(x-b))\big)^2}.
\end{equation}

We will show that in the long run the diagonal is exponentially
repelling. That is, the following theorem holds.

\begin{theorem}\label{diagrep}
For every $a>0$ and $b\in(0,1)$ there exist a positive integer $N$ and
a number $\kappa>1$ such that for every $x\in[0,1]$ we have
\begin{equation}\label{perp}
\prod_{k=0}^{N-1}\lambda_{a,b}(f^k_{a,b}(x))\ge\kappa.
\end{equation}
\end{theorem}

\begin{proof}
We can rewrite~\eqref{derper} as
\[
\lambda_{a,b}(x)=(1+ax(1-x))\exp(a(x-b))
\left(\frac{f_{a,b}(x)}x\right)^2.
\]

Take $\delta>0$ such that $I_{a,b}\subset (\delta,1-\delta)$. If
$\delta$ is sufficiently small, then the interval $(\delta,1-\delta)$
is invariant. Assume that $x\in(\delta,1-\delta)$. Then
$1+ax(1-x)>1+a\delta^2$. Therefore
\[
\lambda_{a,b}(x)>(1+a\delta^2)\exp(a(x-b))
\left(\frac{f_{a,b}(x)}x\right)^2.
\]
Taking the product over a piece of the trajectory of $x$, and using
notation $f^k_{a,b}(x)=x_k$, we get
\[
\prod_{k=0}^{n-1}\lambda_{a,b}(x_k)>(1+a\delta^2)^n
\exp\left(a\sum_{k=0}^{n-1}(x_k-b)\right)\left(\frac{x_n}{x_0}\right)^2.
\]
Using~\eqref{est1}, we get
\begin{equation}\label{exp}
\prod_{k=0}^{n-1}\lambda_{a,b}(x_k)>\delta^4(1+a\delta^2)^n.
\end{equation}

We have $\lambda_{a,b}(0)=\exp(ab)>1$ and $\lambda_{a,b}(1)=
\exp(a(1-b))>1$. Therefore, if $\delta>0$ is sufficiently small then
$\lambda(x)>1+a\delta^2$ whenever $x\le\delta$ or  $x\ge 1-\delta$.
Clearly, in~\eqref{exp} we can take arbitrarily small $\delta>0$, and
the estimate holds whenever $x\in(\delta,1-\delta)$. This means
that~\eqref{exp} holds for all $x\in[0,1]$.

Now we take $N$ such large that $\delta^4(1+a\delta^2)^N>1$ and set
$\kappa=\delta^4(1+a\delta^2)^N$. Then~\eqref{perp} holds.
\end{proof}

\begin{corollary}\label{diagrep1}
There is a neighborhood $U$ of the diagonal such that if $(x,y)\in U$
then the distance of $F_{a,b}^N(x,y)$ from the diagonal is larger than
the distance of $(x,y)$ from the diagonal.
\end{corollary}

Let us fix $a>0$ and $b\in(0,1)$. For simplicity, we will now write
$F$ for $F_{a,b}$ and $(\xhat,\yhat)$ for $F(x,y)$ (if we apply $F$
only once). If we iterate $F$, we will continue to write $(x_n,y_n)$
for $F^n(x,y)$.

Let us establish some simple properties of $F$. We will consider the
triangle below the diagonal in the square $[0,1]^2$. Switching $x$ and
$y$ will result in switching $\xhat$ and $\yhat$, so the corresponding
results for the triangle above the diagonal will be the same.

Consider the rectangle
\[
V=\{(x,y):0\le y<b<x\le 1\}
\]
(see Figure~\ref{regions}).

\begin{lemma}\label{quadrant}
If $(x,y)\in V$ then its trajectory converges to $(1,0)$.
\end{lemma}

\begin{proof}
{}From the formula for $F$ (equation~\eqref{map2d}) it follows that if
$(x,y)\in V$ then $\xhat>x$ and $\yhat<y$. In particular, also
$(\xhat,\yhat)\in V$. Therefore the sequence $(x_n)$ increases, while
the sequence $(y_n)$ decreases. The only possible limits for those
sequences are 1 and 0 respectively.
\end{proof}

Elementary calculations result in the next property.

\begin{lemma}\label{triangle}
The map $F$ preserves the triangle below the diagonal. That is, if
$x>y$ then $\xhat>\yhat$.
\end{lemma}

For small $\delta>0$ we set
\[
T_\delta=\{(x,y):b\le y<x\le1-\delta\}\cup\{(x,y):\delta\le y<x\le
b\}
\]
(see Figure~\ref{regions}).

\begin{figure}
\begin{center}
\includegraphics[height=60truemm]{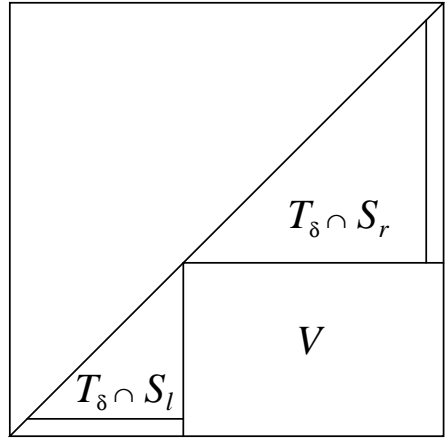}
\caption{Regions $V$, $T_\delta$, $S_r$ and $S_\ell$.}\label{regions}
\end{center}
\end{figure}

\begin{lemma}\label{invreg}
If $\delta>0$ is sufficiently small then $F(T_\delta)\subset
T_\delta\cup V$.
\end{lemma}

\begin{proof}
Set
\[
S_r=\{(x,y):b\le y\le x\le 1\},\ \ \ S_\ell=\{(x,y):0\le y\le x\le b\}.
\]
If $(x,y)\in S_r$ then
\[
\xhat=\frac{x}{x+(1-x)\exp(a(y-b))}\le\frac{x}{x+(1-x)}=x.
\]
Similarly, $\yhat\le y$. The set $F(S_r)$ is a compact subset of the
triangle $\{(x,y):x\ge y\}$, disjoint from the line $y=0$ (because if
$\yhat=0$ then $y=0$). Thus, there is $\delta_r>0$ such that if
$(x,y)\in S_r$ then $\yhat>\delta_r$. By similar arguments, there is
$\delta_\ell>0$ such that if $(x,y)\in S_\ell$ then
$\xhat<1-\delta_\ell$.

Take $\delta\in(0,\min(\delta_r,\delta_\ell))$ and $(x,y)\in
T_\delta\cap S_r$. Then $\xhat\le x\le 1-\delta$. Since $x>y$, by
Lemma~\ref{triangle} we have $\xhat>\yhat$. Moreover,
$\yhat>\delta_r>\delta$. Therefore, $(\xhat,\yhat)\in T_\delta\cup V$.
Similarly, if $(x,y)\in T_\delta\cap S_\ell$ then $(\xhat,\yhat)\in
T_\delta\cup V$. Since $T_\delta\subset S_r\cup S_\ell$, this
completes the proof.
\end{proof}

Similarly as in the one-dimensional case, we get by induction the
formulas
\[
\begin{aligned}
x_n&=\frac{x}{x+(1-x)\exp\left(a\sum_{k=0}^{n-1}(y_k-b)\right)},\\
y_n&=\frac{y}{y+(1-y)\exp\left(a\sum_{k=0}^{n-1}(x_k-b)\right)}.
\end{aligned}
\]

Now we can prove the main theorem of this section.

\begin{theorem}\label{conv2d}
If $0\le y<x\le 1$ then the trajectory of $(x,y)$ converges to
$(1,0)$.
\end{theorem}

\begin{proof}
Suppose that $0<y<x<1$, but the trajectory of $(x,y)$ does not
converge to $(1,0)$. By Lemma~\ref{quadrant}, for all $n$ we have
$(x_n,y_n)\notin V$. Choose $\delta>0$ such that $\delta<y<x<1-\delta$
and the inclusion from Lemma~\ref{invreg} holds with this $\delta$.
Then $(x_n,y_n)\in T_\delta$ for all $n$. Therefore we have
$\delta<y_n<x_n<1-\delta$, so
\[
\delta<\frac{y}{y+(1-y)\exp\left(a\sum_{k=0}^{n-1}(x_k-b)\right)}
<\frac{x}{x+(1-x)\exp\left(a\sum_{k=0}^{n-1}(y_k-b)\right)}<1-\delta.
\]
In the same way as in the proof of Theorem~\ref{t:Cesaro}, we get
\[
\left|\sum_{k=0}^{n-1}(x_k-b)\right|<\frac{2\log(1/\delta)}{a}
\ \ \ \textrm{and}\ \ \
\left|\sum_{k=0}^{n-1}(y_k-b)\right|<\frac{2\log(1/\delta)}{a}.
\]
Therefore,
\[
\sum_{k=0}^{n-1}(x_k-y_k)<\frac{4\log(1/\delta)}{a}.
\]
Since $x_k>y_k$ for all $k$, this means that the series
\[
\sum_{k=0}^\infty(x_k-y_k)
\]
converges, so $\lim_{k\to\infty}(x_k-y_k)=0$. However, this
contradicts Corollary~\ref{diagrep1}.

This proves that in the case $0<y<x<1$ the sequence $(x_n,y_n)$
converges to $(1,0)$.

If $y=0$ then $y_n=0$ for all $n$, and the sequence $(x_n)$ increases.
The limit of this sequence is a number $z>0$ such that $(z,0)$ is a
fixed point of $F$, and the only such point is $(1,0)$. Similarly, of
$x=1$ then $(x_n,y_n)$ converges to $(1,0)$.
\end{proof}

Taking into account what we observed about switching $x$ and $y$, we
can state a more general corollary.

\begin{corollary}\label{c:conv2d}
If a point $(x,y)\in[0,1]^2$ is below (respectively, above) the
diagonal, its $F_{a,b}$-trajectory converges to $(1,0)$ (respectively,
$(0,1)$).
\end{corollary}

\section*{Funding}
Thiparat Chotibut gratefully acknowledges the startup research grant
SRES15111, and the SUTD-ZJU collaboration research grant ZJURP1600103.
Fryderyk Falniowski gratefully acknowledges the support of the
National Science Centre, Poland, grant 2016/21/D/HS4/01798 
and COST Action CA16228 ``European Network for Game Theory''.
Research of Micha{\l} Misiurewicz was partially supported by grant
number 426602 from the Simons Foundation.
Georgios Piliouras was partially supported by SUTD grant SRG ESD 2015
097, MOE AcRF Tier 2 Grant 2016-T2-1-170 and NRF 2018 Fellowship
NRF-NRFF2018-07.

\end{document}